\documentclass[12pt]{article}

\usepackage{amsmath, amssymb,amscd,amsthm,amsfonts}
\usepackage{bm, mathrsfs,dsfont}
\usepackage{float}
\usepackage{tikz}
\usepackage{bbding}
\usepackage{hyperref}

\newtheorem{thm}{Theorem}[section]
\newtheorem{lem}[thm]{Lemma}
\newtheorem{defn}[thm]{Definition}

\newtheorem{rem}[thm]{Remark}
\newtheorem{cor}[thm]{Corollary}

\theoremstyle{definition}
\newtheorem{ex}[thm]{Example}
\newtheorem{qu}[thm]{Question}

\def\qd{quadratic differential}
\def\nd{\noindent}

\title{\bf Drawing cone spherical metrics via Strebel differentials}
\author{Jijian Song$^1$, Yiran Cheng$^2$, Bo Li$^3$ and Bin Xu$^4$}

\begin{document}
\maketitle
\noindent{\small $^{1,3,4}$Wu Wen-Tsun Laboratory of Math, USTC, Chinese Academy of Sciences\\
School of Mathematical Sciences, University of Science and Technology of China\\ No. 96 Jinzhai Road, Hefei, Anhui Province 230026 P. R. China. \\
$^2$ Mathematisches Institut der Universit\"{a}t Bonn\\ Endenicher Allee 60, 53115 Bonn, Germany.
}\\

\noindent{\small \it Correspondence to be sent to: e-mail: bxu@ustc.edu.cn}
\vskip0.5cm

\noindent{\small {\bf Abstract:} {\it Cone spherical metrics} are  conformal metrics with  constant curvature one and finitely many conical singularities on compact Riemann surfaces.  By using Strebel differentials as a bridge, we construct a new class of cone spherical metrics on compact Riemann surfaces by drawing on the surfaces some class of connected metric ribbon graphs.


}

\section{Introduction}
There have been lots of deep applications of Jenkins-Strebel quadratic differentials (JS differentials) to the moduli space of $n$-pointed compact Riemann surfaces due to Mumford and Thurston (unpublished works), Harer \cite{Har86, Har88} and Harer-Zagier \cite{HZ86}, Kontsevich \cite{Kon92}, Looijenga \cite{Loo93, Loo95} and Hain-Looijenga \cite{HL97} et al. since Jenkins \cite{Jen78} and Strebel \cite{Strebel} proved their celebrated existence theorems about the differentials.  These applications have closed interrelation with singular conformal metrics of constant curvature $0$ defined by JS differentials. In this manuscript we apply JS differentials to the study of cone spherical metrics, which is also relevant to both Onsager's vortex model in statistical physics (\cite{CLMP92}) and the Chern-Simons-Higgs equation in superconductivity (\cite{CLW2004}).

A {\it cone spherical/hyperbolic/flat metric} (or {\it spherical/hyperbolic/flat metric} for short in the following context) is a conformal metric with constant curvature $(+1)/(-1)/0$ and finitely many conical singularities on a compact Riemann surface. A classical question asks what is the necessary and sufficient condition under which there exists a conformal metric of constant curvauture $K\in \{+1,-1,0\}$ with the prescribed conical singularities on a given compact Riemann surface. Its history goes back to {\' E}. Picard \cite{Pi1905} and H. Poincar{\' e} \cite{Po1898}.  The question can be thought of as a generalization of the celebrated Uniformization Theorem. There is a well known necessary condition for the existence problem which is given by the Gauss-Bonnet formula (Subsection \ref{subsec:elliptic}). McOwen \cite{MO88} and Troyanov \cite{Tr91} proved that this condition is also sufficient for the existence of hyperbolic and flat metrics. Moreover, they also showed that the corresponding hyperbolic or flat metric exists uniquely with the prescribed singularities on a given Riemann surface.  However, it is not the case for spherical metrics. The question in this case is still widely open nowadays. We think that the complexities of spherical metrics at least consist of the following two interconnected aspects:\\

$\noindent$ $\bullet$ Spherical metrics with the prescribed conical singularities are not unique as they have conical singularities of angles greater than $2\pi$ \cite{LT92, Tr89} and their existence does not only depend on the conical angles of singularities, but also rely on the configuration of the singularities \cite{CL14, CWWX2015}. Moreover, there exist blow-up phenomena for the corresponding nonlinear spherical PDE of spherical metrics \cite{CL14} so that we could not expect a priori $C^0$ estimate for this PDE in general. It is well known that the PDE does not admit the $C^0$ estimate as there exist noncompact families of reducible spherical metrics with the prescribed singularities \cite{CWWX2015, CKL2016}. The last author conjectures that the PDE should always admit the $C^0$ estimate except the reducible case.   \\

$\noindent$ $\bullet$ No general conjecture has been proposed though people have speculated that the existence of spherical metrics may be equivalent to some algebraic stability motivated by the famous Yau-Tian-Donaldson conjecture in K{\" a}hler Geometry. Quite recently, the first, the last authors and L. Li  \cite{SXL2017} found a correspondence between  spherical metrics with conical angles in $2\pi {\Bbb Z}_{>1}$ and line subbundles of rank two polystable vector bundles on compact Riemann surfaces, and speculated  that the correspondence there could be be generalized to a more general one valid for all spherical metrics by where the notion of parabolic vector bundle \cite{MS80} will be involved. \\

\noindent On the other hand, substantial progress and deep understanding \cite{UY2000, Er04, BdMM11, EGT1405, EG15, EGT1409, EGT1504, CLW14, MP1505, CWWX2015} have been achieved due to the effort of many mathematicians.

To better understand spherical metrics, we shall continue to construct explicit examples in this manuscript after \cite{CWWX2015}, by studying the interrelation between spherical metrics and a special class of JS differentials, called {\it Strebel differentials}. Their definitions will be given in Subsection \ref{subsec:JS}. We summarize in the following the main results of this manuscript, whose details will be given in the subsections of this introductory section.\\

\noindent {\it We at first make the key observation that a Strebel differential $q$ has real periods {\rm (Definition \ref{def:real} and Lemma \ref{lem:RealPeriod})}  so that the multi-valued function
\[ \exp\bigg(\int^z \sqrt{-q}\biggr) \]
over the Riemann surface punctured by both the zeroes and the poles of $q$ turns out to be a developing map of a cone spherical metric {\rm (Theorem \ref{thm:StrebelMetric}).}  Then we construct a new class of spherical metrics {\rm (Corollaries \ref{cor:Strebel} and \ref{cor:quasi})} by using this observation and the well known one-to-one correspondence between the set of Strebel differentials and that of certain connected metric ribbon graphs} \cite{Har88}.\\

\noindent We are motivated by Strebel's celebrated book \cite{Strebel}, the stimulating paper \cite{Mul98} by Mulase-Penkava and \cite{CWWX2015}. Q. Chen, W. Wang, Y. Wu and the last author used in \cite{CWWX2015} a meromorphic 1-form $\omega$ with simple poles and real periods to construct developing maps $\lambda\cdot\exp\bigl(\int^z \sqrt{-1}\,\omega\bigr)$ with $\lambda>0$ of reducible spherical metrics. This manuscript may be thought of as an effective and powerful expansion of \cite{CWWX2015} in the sense that the former not only generalizes the latter, but also establishes a new connection between cone spherical metrics on Riemann surfaces and metric ribbon graphs on the underlying topological surfaces.

\subsection{JS and Strebel differentials}
\label{subsec:JS}
Let $X$ be a compact connected Riemann surface of genus $g\geq 0$ and $K_X$ its canonical line bundle.  We call a meromorphic section of $K_X^{\otimes 2}$ a {\it quadratic differential}  on $X$, and call either a zero or a pole of a quadratic differential $q$ a {\it critical point} of $q$. We denote by ${\rm Crit}(q)$ the set of critical points of $q$. {\it In this manuscript we only consider quadratic differentials whose crtical points have order at least $(-2)$, i.e.  which have at most double poles, if any.} Locally, in a simply connected neighborhood of a non-critical point, $q$ can be expressed as a square of an Abelian differential, but globally it is not the case in general. We call a \qd\ $q$ {\it {\rm (}non{\rm )}-degenerate} if it is (not) globally the square of an Abelian differential.

Such a quadratic differential $q$ defines a conformal flat metric, called the $q$-metric, on $X \setminus {\rm Crit}(q)$ such that each critical point $P$ of $q$ with order $\geq (-1)$ is a conical singularity  of the metric with the angle of $\pi\,\bigl(2+{\rm ord}_P(q)\bigr),$ and the metric forms a semi-infinite cylinder near each critical point of order $(-2)$, i.e. each second order pole of $q$.

Using the similar process in \cite[Construction in Section 2]{Lann2004}, we obtain from the pair $(X,q)$ the {\it spectral cover} $\pi:{\hat X}\to X$, which is a canonical double cover branched exactly over all the critical points of $q$ with odd order, such that there exists a globally defined meromorphic 1-form $\omega$ with at most simple poles on ${\hat X}$ and $\pi^*(q)=\omega \otimes \omega.$ Moreover, ${\hat X}$ is called the {\it spectral curve of $q$}, which is connected if and only if $q$ is non-degenerate. Denote by $\tau$ the holomorphic involution of the double cover $\pi:{\hat X}\to X$. Then $\omega$ is anti-invariant under the action of $\tau$, $\tau^*\omega=-\omega$. We call a multi-valued meromorphic function $f$ on a Riemann surface $S$, not necessarily compact, {\it a projective function} if each branch of $f$ is locally univalent and the monodromy representation of $f$ gives a homomorphism from $\pi_1(S,\,B)$ to the group ${\rm PGL}(2,{\Bbb C})$ consisting of all M{\" o}bius transformations, where $B\in S$ is an arbitrarily chosen base point. Then $f(z)=\exp\,\bigl(\int^z\, \sqrt{-1}\,\omega\bigr)$ is a projective function on ${\hat X} \setminus {\rm Crit}(\omega)$ with monodromy in ${\Bbb C}^*:=\left\{z\mapsto e^{w}z:w\in {\Bbb C}\right\}$, the projective function $F(z):=\exp\,\bigl(\int^z\, \sqrt{-q}\bigr)$ on $X\setminus {\rm Crit}(q)$ has monodromy in ${\Bbb C}^*\rtimes {\Bbb Z}_2:=\langle \psi,\,z\mapsto\frac{1}{z}:\,\psi\in {\Bbb C}^*\rangle$, and $\pi^*F=F\circ \pi=f$.

\begin{defn}
\label{def:real}
{\rm Use the notions in the above paragraph. We say that $q$ {\it has real periods} if and only if $\omega$ has real periods, i.e.
$\int_\gamma\, \omega\in {\Bbb R}$, where $\gamma$ is a piecewise smooth loop in the surface of ${\hat X}$ punctured by the simple poles of $\omega$.}
\end{defn}

In a complex local coordinate chart $(U,\,z)$ of $X$, if $q$ has form $q=f_U(z)\,dz^2$, then the $q$-metric has expression $|f_U|\,|dz|^2$. The \textit{horizontal geodesics} of the $q$-metric are defined to be curves $\gamma=\gamma(t)$ on $X$ satisfying $f_U\big(\gamma(t)\big)\gamma '(t)^2 > 0$ for every $t$ in $\gamma$'s domain. A \textit{horizontal trajectory} (\textit{trajectory} for short in the following context) of the differential $q$ is a maximal horizontal geodesic, i.e. is not contained in a longer curve which is also a horizontal geodesic. A trajectory may be a closed curve (\textit{closed}), or connecting two critical points of $q$ (\textit{critical}), or neither (\textit{recurrent}). It was shown by Strebel \cite{Strebel} that the closure of a recurrent trajectory as a subset of $X$ has positive measure.

By a \textit{JS differential} on $X$ \cite[Section 5]{AC2010} we refer to a \qd\ on $X$ such that\\

\noindent $\bullet$ it has no recurrent trajectories, and

\noindent $\bullet$ it could be expressed as
\[ - \bigg(\frac{b_{-2}}{z^2} + \frac{b_{-1}}{z} + b_0 + b_1z + \cdots\bigg) \,dz^2\quad {\rm with}\quad b_{-2} > 0 \]
in a complex local coordinate chart $z$ centered at a double pole of it.\\

\noindent Hence, all of the trajectories of a JS differential are either closed or critical. The \textit{critical graph} of a Jenkins-Strebel differential $q$ is the metric ribbon graph on the surface consisting of its critical trajectories whose lengths are induced by the $q$-metric and is not necessarily connected.

By a {\it Strebel differential} on $X$ we refer to a JS differential on $X$ which has $n\geq 1$ double poles and has no simple pole, whose trajectories have finite lengths, and whose critical graph is connected and decomposes $X$ into $n$ open topological discs, which contain the $n$ poles, respectively. {\it Later on, we only consider connected metric ribbons graphs {\rm (}metric ribbon graphs {\rm in short)} on compact oriented topological surfaces which decompose the surfaces into several discs.} Near one of its double poles, a Strebel differential has expression $-\bigl(b_{-2}z^{-2}+b_{-1}z^{-1}+b_0+b_1z+\cdots\bigr)dz^2$, where $b_{-2}>0$ is a constant independent of the coordinate chart $z$ centered at the pole, and we call the positive number $\sqrt{b_{-2}}$ the {\it residue} of the differential at the pole. Let $P_1, P_2, \cdots, P_n$ be all the double poles of a Strebel differential $q$ on $X$ such that $(2-2g-n)<0$ and $a_j>0$ be the residue of $q$ at $P_j$ for $j=1, 2, \cdots, n$. Then we call $\alpha(q) = (a_1, a_2, \cdots, a_n)$ the {\it residue vector} of $q$. We call the partition $\lambda(q)=(m_1,\cdots,m_\ell)$ of $(2n+4g-4)>0$ formed by the orders of $m_1,\cdots, m_\ell$ of the zeroes $Q_1,\cdots, Q_\ell$ of $q$  the {\it zero partition} of $q$. We call the following ${\Bbb R}$-divisor
\[ D_q=\sum_{j=1}^n\, (a_j-1)\,P_j+\sum_{k=1}^\ell\, \frac{m_k}{2}\, Q_k \]
the {\it spherical divisor associated with} $q$.

Given $n\geq 1$ marked points $P_1,\cdots, P_n$ on $X$ such that $2-2g-n<0$ and $n$ positive numbers $a_1,\cdots, a_n$, Strebel \cite[Theorem 23.5]{Strebel} proved that there exists a unique Strebel differential on $X$ which has double poles at $P_1,\cdots, P_n$ with residues $a_1,\cdots, a_n$. Harer \cite{Har86} proved that there exists a one-to-one correspondence between the set of metric ribbon graphs and the set of Strebel differentials.  Hence, in order to prove the existence of some special class of Strebel differentials, we only need to construct the corresponding metric ribbon graphs.

The following observation forms a cornerstone of the construction of spherical metrics in this manuscript.\\

\noindent {\bf The key observation} (Lemma \ref{lem:RealPeriod}) {\it  Each Strebel differential has real periods.}

\subsection{Cone spherical metrics}
\label{subsec:elliptic}
Let $D=\sum_{j=1}^n\,(\theta_j-1)\,P_j$ be a ${\Bbb R}$-divisor on a compact connected Riemann surface $X$ such that $\theta_j>0$. We call that a conformal metric $ds^{2}$ on $X$ {\it represents} $D$ or has a {\it conical singularity} at $P_j$ with {\it cone angle} $2\pi\theta_j>0$ for all $j=1,2,\cdots,n$ if $ds^{2}$ is a smooth conformal metric on $X\setminus {\rm supp}\,D=X\setminus\{P_1,\cdots, P_n\}$ and in a neighborhood $U_j$ of $P_j$, $ds^{2}$ has form $e^{2u_{j}}\,|dz|^2$, where $z$ is a local complex coordinate defined in $U_j$ centered at $P_j$ and the real valued function $u_{j}-(\theta_j-1)\,\ln\,|z|$ lies in the space $C^0(U_j)\cap C^\infty\bigl(U_j\setminus \{P_j\}\bigr)$. If the (Gaussian) curvature $K_{ds^{2}}$ of $ds^{2}$ satisfies some mild regularity condition, there holds the Gauss-Bonnet formula (\cite[Proposition 1]{Tr91}),
\[ \int_{\Sigma^*}\, K_{ds^{2}}\, {\rm d}V_{ds^{2}} = 2\pi\,\chi(X,\,D):=2\pi\,\Bigl(\chi(X)+\deg\,D\Bigr), \]
where we denote by $\chi(X)$ the Euler number of $X$, and by $\deg\,D=\sum_{j=1}^n\, (\theta_j-1)$ the degree of  $D$.

As mentioned before, $\chi(X,\,D)=0$ ($<0$) forms the necessary and sufficient condition for the existence of a cone flat (hyperbolic) metric representing $D$. However, $\chi(X,\,D)>0$ is only a necessary condition for the existence of spherical metrics representing $D$ and is not sufficient at all. Since a conical singularity $P$ of angle $2\pi$ of a conformal metric with constant curvature is actually a smooth point of the metric (\cite[Proposition 3.6]{CWWX2015}), we may look at it as a {\it marked} smooth point of the metric. Since the existence problem of spherical metrics is widely open, it is meaningful to find more examples for the better understanding of spherical metrics. Using the viewpoint of Complex Analysis, we shall construct new examples of cone spherical metrics in terms of Strebel differentials. We need prepare the notion of developing map at first.

Let $f$ be a projective function on a Riemann surface $S$. Then in a local complex coordinate chart $(U,\,z)$ of $S$, the Schwarzian derivative $\{f,\,z\}=\bigl(\frac{f''(z)}{f'(z)}\bigr)' - \frac{1}{2}\bigl(\frac{f''(z)}{f'(z)}\bigr)^2$ of $f$ is a single valued meromorphic function on $(U,z)$. We call that a projective function $f$ on $X\setminus {\rm supp}\, D$ {\it is compatible with the real divisor $D=\sum_{j=1}^n\, (\theta_j-1)P_j$ on $X$} if for each $1\leq j\leq n$ there exists a complex coordinate chart $(U_j,\,z_j)$ of $X$ centered at $P_j$ such that the meromorphic function $\{f,\,z_j\}$ has the principal singular part of $\frac{1-\theta_j^2}{2}\,z_j^{-2}$. We call a projective function $f$ on $S$ has {\it unitary monodromy} if it has monodromy in
\[ {\rm PSU}(2)=\left\{z\mapsto \frac{az+b}{-\overline{b}z+\overline{a}}:|a|^2+|b|^2=1\right\}\subset {\rm PGL}(2,\,{\Bbb C}).\]
It was proved in \cite[Theorem 3.4]{CWWX2015} that there exists an spherical metric $ds^2$ on $X$ representing $D$ if and only if there is a projective function $F$ on $X \setminus {\rm supp} \, D$ which is compatible with $D$ and has unitary monodromy. Moreover, $ds^2$ can be thought of as the pull-back $F^*(g_{\rm st})$ by $F$ of the standard metric $g_{\rm st}=\frac{4|dw|^2}{(1+|w|^2)^2}$ on the Riemann sphere $\overline{\Bbb C}={\Bbb C}\cup\{\infty\}$. We call $F$ a {\it developing map} of the spherical metric $ds^2$. We call an spherical metric $ds^2$ {\it reducible} if it has a developing map $F$ whose monodromy falls into ${\rm U}(1)=\left\{z\mapsto e^{\sqrt{-1}t}z:t\in [0,\,2\pi)\right\}$; otherwise, we call $ds^2$ {\it irreducible} (\cite[p.76]{UY2000}). Reducible (irreducible) metrics are also called cone spherical metrics with {\it co-axial {\rm(}non-coaxial{\rm )} holonomy} in the literature such as \cite{MP1505, Er2017}.    For simplicity, we may regard (the conjugacy class of) the monodromy of a developing map of an spherical metric as the {\it monodromy of the metric} (\cite[Lemma 2.2]{CWWX2015}).

By the key observation in Subsection \ref{subsec:JS}, we can construct spherical metrics in terms of Strebel differentials in the following

\begin{thm}
\label{thm:StrebelMetric}
Let $q$ be a Strebel differential on a compact connected Riemann surface $X$ and $D_q$ the spherical divisor associated with $q$. Then the projective function $F(z)=\exp\big(\int^z \sqrt{-q}\big)$ on $X\setminus {\rm supp}\, D_q$ has monodromy in $${\rm U}(1)\rtimes {\Bbb Z}_2:=\langle \psi,\,z\mapsto 1/z:\psi\in {\rm U}(1) \rangle\subset {\rm PSU}(2)$$ and is compatible with $D_q$. In particular, $F$ is a developing map of a cone spherical metric $ds^2$ representing $D_q$.
\end{thm}

\begin{rem}
\label{rem:red}
{\rm If $q$ is globally the square of an Abelian differential $\omega$, then the theorem in this case reduces to \cite[Theorem 1.5]{CWWX2015}. It was shown there that $\omega$ could actually give a one-parameter family of reducible spherical metrics on $X$ by the one-parameter family $\left\{f_\lambda(z)=\lambda\cdot\exp\,\bigl(\int^z\,\sqrt{-1}\,\omega\bigr):\lambda>0\right\}$ of developing maps. }
\end{rem}

Using Theorem \ref{thm:StrebelMetric} and the existence theorem of Strebel differentials (Theorem \ref{thm:Strebel}), we obtain the following existence result of spherical metrics.

\begin{cor}
\label{cor:Strebel}
Given any $n\geq 1$ positive numbers $a_1,\cdots, a_n$ and any $n$ points $P_1,\cdots, P_n$ on a compact connected Riemann surface $X$ of genus $g$ such that $(2-2g-n)<0$, there exists an spherical metric $ds^2$ on $X$ which has conical angles of $2\pi a_j$ at $P_j$ for $1\leq j\leq n$ and $\ell$ extra conical singularities $Q_1,\cdots, Q_\ell$ which have angles $(2+m_1)\pi,\cdots, (2+m_\ell)\pi$, respectively, such that $m_1,\cdots, m_\ell$ are positive integers of sum $(2n+4g-4)$. Note that each point $P_j$ with $a_j=1$ is a marked smooth point of $ds^2$.
\end{cor}

\begin{rem}
{\rm A more general version of Theorem \ref{thm:StrebelMetric} and Corollary \ref{cor:Strebel} also holds true as we replace Strebel differentials by JS differentials with real periods and having at least a double pole. The proof is quite similar. The crucially valuable point of the theorem and the corollary is that we could use directly Strebel differentials, or, equivalently, their metric ribbon graphs to construct cone spherical metrics. This justifies the title of this manuscript.

B. Li and the last author \cite{LX2017} proved a relevant statement that an irreducible spherical metric with monodromy in ${\rm U}(1)\rtimes {\Bbb Z}_2$ has a developing map of form $\exp\big(\int^z \sqrt{-\phi}\big)$, where $\phi$ is a non-degenerate JS differential with real periods and having at least a double pole.}
\end{rem}

Motivated by Theorem \ref{thm:StrebelMetric} and Corollary \ref{cor:Strebel}, we propose the following definition and question.

\begin{defn}
\label{def:res}
{\rm Let $g\geq 0$ and $n>0$ be integers such that $(2-2g-n)<0$. We say a vector $\alpha=(a_1,\cdots, a_n)\in ({\Bbb R}_{>0})^n$ and a partition $\lambda=(m_1,\cdots, m_\ell)$ of $(2n+4g-4)$ are {\it compatible} if there exist a compact connected Riemann surface $X$ of genus $g$ and a Strebel differential $q$ on $X$ such that $q$ has residue vector $\alpha$ and zero partition $\lambda$.}
\end{defn}

\begin{qu}
What is the necessary and sufficient condition under which a vector in $({\Bbb R}_{>0})^n$ and a partition of $(2n+4g-4)$ are compatible?
\end{qu}

By drawing corresponding {\it trivalent metric ribbon graphs} (Definition \ref{defn:tri}), we could answer this question for the simple partition $(1,\cdots, 1)$ in the following

\begin{thm}
\label{thm:SimpleZeros}
Let $g \geq 0$ and $n > 0$ be integers such that $2-2g-n<0$ and $a_1,\cdots, a_n$ be $n$ positive numbers
such that $a_1\leq a_2\leq \cdots \leq a_n$. Suppose that $a_1+a_2 \neq a_3$ if $(g,n)=(0,3)$.
The residue vector $(a_1,\cdots, a_n)$ and the simple partition $(1,\cdots, 1)$ of $(2n+4g-4)$ are compatible.\end{thm}

\begin{cor}
\label{cor:quasi}
Under all the assumptions of Theorem \ref{thm:SimpleZeros},  there exist a compact connected Riemann surface $X$ of genus $g$ and an spherical metric  on $X$ which has exactly $(4g-4+3n)$ conical singularities with the following angles
\[ 2\pi a_1,\,2\pi a_2,\,\cdots, 2\pi a_n,\, \underbrace{3\pi,\,3\pi,\cdots,3\pi}_{4g-4+2n}. \]
\end{cor}

The organization of this manuscript is as follows. In Section $2$, for the convenience of readers, we review briefly the existence theorem of Strebel differentials and the correspondence between Strebel differentials and metric ribbon graphs, where we recall a well known example of Strebel differentials with three double poles on the Riemann sphere $\overline{{\Bbb C}}$. In Section \ref{sec:StrebelMetric}, we at first show the key observation, from which we then obtain the proof of  Theorem \ref{thm:StrebelMetric}. As an application of the theorem, we give an example of cone spherical metrics on the two elliptic curves with extra symmetry in this section. The proof of Theorem \ref{thm:SimpleZeros} occupies the whole of Section \ref{sec:SimpleZeros}. In order to prove the existence of such Strebel differentials, we only need to construct the corresponding trivalent metric ribbon graphs. Our strategy is to construct them on the two-shere at first and then on the surfaces of positive genus by attaching handles. We discuss in Section \ref{sec:sphere} the properties for the cone angles of cone spherical metrics on the Riemann sphere in Corollaries \ref{cor:Strebel} and \ref{cor:quasi} which are relevant to  \cite{MP1505,Er2017}.

\section{Strebel differentials and metric ribbon graphs}
\label{sec:Strebel}
In this section, we will recall in detail the existence theorem of Strebel differentials and the correspondence between metric ribbon graphs and Strebel differentials.

\begin{thm}
\label{thm:Strebel}
{\rm (\cite[Theorem 23.5]{Strebel})}
Given a compact connected Riemann surface $X$ of genus $g\geq 0$ with $n\geq 1$ marked points $P_1, P_2, \cdots, P_n$ such that $2 - 2g - n < 0$ and  $n \geq 1$ positive real numbers $a_1,\cdots, a_n$,  there exists a unique quadratic meromorphic differential $q$ on $X$ satisfying the following conditions:
  \begin{enumerate}
    \item $q$ has a double pole at each marked point $P_j$ with residue $a_j$ and it has no other poles.
    \item Every closed trajectory of $q$ circles around exactly one of the marked points.
    \item $q$ has no recurrent trajectory.
  \end{enumerate}
\end{thm}

We called such a differential $q$  a {\it Strebel differential} in the introduction. Consider a double pole $P_j$ of $q$. Then, by \cite[Theorem 7.2 and Fig. 9]{Strebel}, $q$ has form $ -\frac{a_{j}^2}{z^2}\, dz^{2}$, the trajectories of $q$ near $P_{j}$ are all closed. The union of these closed trajectories circling around $P_{j}$ and $\{P_{j}\}$ forms an open disk $D_j$ centered at $P_j$, whose boundary $\partial D_j$ consists of some critical trajectories of $q$ and has the length of $2\pi a_j$.  The last two conditions of Theorem \ref{thm:Strebel} are equivalent to
\begin{align}
\label{decomposition}
 X = \bigcup_j \overline{D}_j
\end{align}
Every critical trajectory of $q$ is bounded by zeroes of $q$.  There are $(m+2)$ half critical trajectories incident to a zero of order $m$ of $q$. In addition, the {\it critical graph} of $q$ is defined to be the connected graph, whose vertices coincide with the zeroes of $q$, and whose edges coincide with  the critical trajectories with lengths induced by the $q$-metric. In a nutshell, the Strebel differential $q$ on $X$ defines a metric ribbon graph $\Gamma=\Gamma(q)$ on the underlying oriented topological surface $\mathbb{X}$ such that
\begin{itemize}
  \item At each vertex of $\Gamma(q)$ there is a natural cyclic order on the set of half edges (half critical trajectories of $q$) incident to the vertex defined by the orientation of $\mathbb{X}$.
  \item each vertex of $\Gamma(q)$ has valency at least three.
  \item $\Gamma(q)$ decomposes $\mathbb{X}$ into the disjoint union of $n$ open topological discs.
\end{itemize}

\noindent On the other hand,  Harer \cite{Har86} and Mulase-Penkava \cite[Theorem 5.1]{Mul98} proved that given such a metric ribbon graph $\Gamma$ on an oriented compact topological surface $\mathbb{X}$, there exists a {\it unique} Riemann surface structure $X$ over $\mathbb{X}$ and a {\it unique} Strebel differential $q$ on $X$ such that $\Gamma$ coincides with the critical graph of $q$. Therefore, there exists a correspondence between Strebel differentials and metric ribbon graphs as above.

Note that for any metric ribbon graph on $\overline{\mathbb{C}} = \mathbb{C} \cup \{ \infty \}$,  we could view it as a planar metric graph with natural cyclic order at each vertex.
\begin{ex}
\label{ex:3Poles}
  Let $q$ be a quadratic differential on $\overline{\mathbb{C}}$ such that it has three double poles at $(0, 1, \infty)$ with residue vector $(a_{1},a_{2} , a_{3}) \in (\mathbb{R}_{>0})^3$. Then $q$ exists uniquely and
\[ q = - \bigg( \frac{a_1^2}{z^2} + \frac{a_2^2}{(z-1)^2} + \frac{a_3^2-a_1^2-a_2^2}{z(z-1)}\bigg) dz^{2}. \]
$q$ must be a Strebel differential by Theorem \ref{thm:Strebel}. Assume that $a_1 \leq a_2 \leq a_3$. It is well known that the metric ribbon graph of $q$ coincides with one of the following three graphs, where $b_{j} = 2\pi a_{j}, j = 1,2,3$.
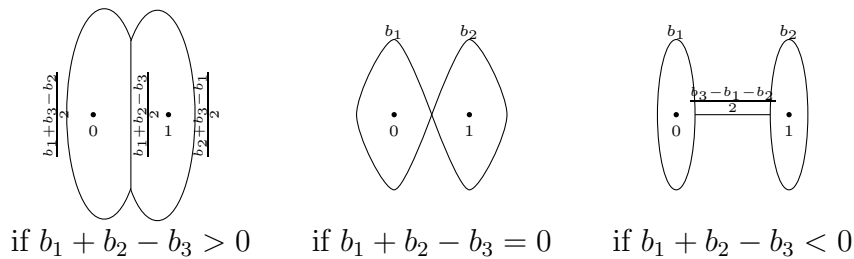
\begin{figure}[H]
  \begin{center}
    \begin{tikzpicture}
      \draw (-4,1)--(-4,-1);
      \draw (-4,1) arc (45:315:0.5cm and 1.4cm);
      \draw (-4,-1) arc (-135:135:0.5cm and 1.4cm);
      \fill (-4.5,0) circle (1pt) node[below] {\tiny{$0$}}
           (-3.5,0) circle (1pt) node[below] {\tiny{$1$}};
      \node[rotate=90] at (-3.8,0) {\tiny{$\frac{b_1+b_2-b_3}{2}$}};
      \node[rotate=90] at (-5,0) {\tiny{$\frac{b_1+b_3-b_2}{2}$}};
      \node[rotate=90] at (-3,0) {\tiny{$\frac{b_2+b_3-b_1}{2}$}};
      \node at (-4,-1.7) {if $b_1 + b_2 - b_3 > 0$};

      \draw plot[smooth] coordinates {(0,0) (-0.5,1) (-1,0) (-0.5,-1) (0,0)};
      \draw plot[smooth] coordinates {(0,0) (0.5,1) (1,0) (0.5,-1) (0,0)};
      \fill (-0.5,0) circle (1pt) node[below] {\tiny{$0$}}
           (0.5,0) circle (1pt) node[below] {\tiny{$1$}};
      \node at (-0.5,1.1) {\tiny{$b_{1}$}};
      \node at (0.5,1.1) {\tiny{$b_{2}$}};
      \node at (0,-1.7) {if $b_{1} + b_{2} - b_{3} = 0$};

      \draw (3.5,0)--(4.5,0);
      \draw (4.75,0) ellipse (0.25 and 1)
            (3.25,0) ellipse (0.25 and 1);
      \fill (3.25,0) circle (1pt) node[below] {\tiny{$0$}}
           (4.75,0) circle (1pt) node[below] {\tiny{$1$}};
      \node at (3.25,1.1) {\tiny{$b_{1}$}};
      \node at (4.75,1.1) {\tiny{$b_{2}$}};
      \node at (4,0.2) {\tiny{$\frac{b_{3}-b_{1}-b_{2}}{2}$}};
      \node at (4,-1.7) {if $b_1 + b_2 - b_3  < 0$};
    \end{tikzpicture}
  \end{center}
\caption{Metric ribbon graphs with $3$ double poles on $\overline{\Bbb C}$.}
\label{fig:3Poles}
\end{figure}

By Theorem \ref{thm:StrebelMetric}, we could construct from the above graphs some cone spherical metrics as follows. The Strebel differentials in the second case are degenerate so that they give a one-parameter family of reducible spherical metrics with the four cone angles of $2\pi a_1,\,2\pi a_2, 2\pi(a_1+a_2)$ and $4\pi$,  which was mentioned in Remark \ref{rem:red}. Those in both the first and third cases give new examples of cone spherical metrics with the five cone angles of $2\pi a_1,\,2\pi a_2, 2\pi a_3,\,3\pi,\,3\pi$, which are irreducible in general. Moreover,  ${\rm Crit}(q)$ coincides with the set of conical singularities of these metrics.
\end{ex}

\section{Proof of Theorem \ref{thm:StrebelMetric}}
\label{sec:StrebelMetric}
The strategy of the proof is to show that the projective function $F(z)=\exp\big(\int^z \sqrt{-q}\big)$ on $X\setminus {\rm Crit}(q)$ has monodromy in ${\rm U}(1)\rtimes {\Bbb Z}_2\subset {\rm PSU}(2)$, which is reduced to the following

\begin{lem}
\label{lem:RealPeriod}
{\rm \bf (The key observation)}  $q$  has real periods.
\end{lem}

\begin{proof}
Using the Riemann existence theorem \cite[Theorem 2, p.49]{Don2011}, we rewrite the construction of the canonical (branched) double cover $\pi:{\hat X}\to X$ induced by the pair $(X,\, q)$ (\cite[Section 2]{Lann2004}) in a global point of view. Denote by $X^{\rm ev}$ the surface of $X$ punctured by the critical points of odd orders of $q$. We look at the monodromy representation of the multi-valued one-form $\sqrt{q}$ on the punctured surface $X^{\rm ev}$ as a permutation representation
\[ \rho=\rho_q:\,\pi_1(X^{\rm ev})\to S_2. \]
For example, we obtain another different branch of $\sqrt{q}$ after doing the analytic continuation of a given branch of $\sqrt{q}$ along a small simple closed curve $\gamma$ around a critical point of odd order of $q$. That is, the image of the homotopy class $[\gamma]$ of $\gamma$ under $\rho$ is non-trivial. However, the non-triviality of $\rho([\Gamma])$ depends on $q$ itself, where $\Gamma$ is a loop lying in $X^{\rm ev}$ and representing a non-trivial element in $\pi_1(X)$. The representation $\rho:\,\pi_1(X^{\rm ev})\to S_2$ is non-trivial if and only if $q$ is non-degenerate. By the Riemann existence theorem, we obtain from $\rho$ the canonical double cover $\pi:{\hat X}\to X$, branched over critical points of odd orders of $q$, such that $\pi^*q=\omega\otimes \omega$, where $\omega$ is a meromorphic one-form on ${\hat X}$. The unique pre-image of a zero of odd order $k$ of $q$ is a zero of order $(k+1)$ of the one-form $\omega$, and the two pre-images of a critical point of even order $k$ are critical points of order $k/2$ of $\omega$. (Actually, by \cite[Theorem 6.1]{Strebel} we have $q=w^k\,dw^2$ in a suitable complex coordinate $w$ centered at a zero $Q$ of odd order $k$. Since $w=\pi(z)=z^2$ near $\pi^{-1}(Q)$ in a suitable complex coordinate $z$ centered at $\pi^{-1}(Q)$, we have $\pi^*\, q=\big(z^2\big)^k\,d(z^2)\otimes d(z^2)=\big(2z^{k+1}\, dz\big)\otimes \big(2z^{k+1}\, dz\big)$ and $\omega=\pm 2z^{k+1}dz$  near $\pi^{-1}(Q)$, which gives the first statement. The latter follows similarly.) In particular, $\omega$ has only simple poles. Moreover, ${\hat X}$ is connected if and only if $q$ is non-degenerate; and $\pi:{\hat X}\to X$ is an un-ramified covering if and only if $q$ has no critical point of odd order.

Since $q$ is a Strebel differential, we could make the following \\

\noindent {\sc Claim} {\it $\pi^*q=\omega\otimes \omega$ is a Strebel differential on ${\hat X}$.  As ${\hat X}$ has two connected components, here we mean  that $\pi^*q$ is a Strebel differential on each component of ${\hat X}$.}

\noindent {\sc Proof of the claim}\quad Denoting ${\hat X}^{\rm ev} = \pi^{-1}(X^{\rm ev})$, we could see that $\pi: {\hat X}^{\rm ev}\to X^{\rm ev}$ is a covering space of degree two.

Then we show that {\it $\pi^*q=\omega\otimes \omega$ is a JS differential, i.e. it has no recurrent trajectory.} Denote by $L$ the maximal value of lengths of horizontal trajectories of $q$. Take an open horizontal geodesic $\xi=\xi(t)$ such that the closure of $\xi$ lies in ${\hat X}\setminus {\rm Crit}(\omega)\subset {\hat X}^{\rm ev}$. Then the image of $\gamma(t)=\pi\circ\xi(t)$ is also a horizontal geodesic segment for the $q$-metric whose closure lies in $X\setminus {\rm Crit}(q)\subset X^{\rm ev}$ and whose length is at most $L$. Since $\pi: {\hat X}^{\rm ev}\to X^{\rm ev}$ is both a degree two covering and a local isometry, the length of $\xi$ is at most $2L$. Therefore, $\pi^*q$ has no recurrent trajectory since each such trajectory has infinite length.

At last, we show that {\it each closed horizontal trajectory of $\pi^*q=\omega\otimes \omega$ circles around exactly one simple pole of $\omega$}. Choose a closed horizontal trajectory $\xi$ of $\pi^*q=\omega\otimes \omega$. Then $\xi$ lies in ${\hat X}\setminus {\rm Crit}(\omega)$. Then the image $\gamma$ of $\xi$ by $\pi$ is a closed horizontal trajectory of $q$, and $\pi:\xi\to\gamma$ forms a covering space of degree $d\leq 2$. Since $q$ is a Strebel differential, there exists a family of closed horizontal trajectories $\gamma_t$ of $q$, $t\in [0,\,\infty)$, where $\gamma=\gamma_0$ and each $\gamma_t$ circles around exactly one double pole of $q$, say $P$. Denote by ${\hat P}_1$ and ${\hat P}_2$ the two pre-images of $P$ under $\pi:{\hat X}\to X$, which are two simple poles of $\omega$. Since $\xi$ is a lifting of $\gamma$ and $\gamma_t$ is a homotopy, by the homotopy lifting property, there exists a unique homotopy $\xi_t$ such that $\xi=\xi_0$, each $\xi_t$ is a lifting of $\gamma_t$ under the covering $\pi: {\hat X}^{\rm ev}\to X^{\rm ev}$, and $\xi_t$ forms a family of closed horizontal trajectories of $\pi^*q$. By the continuity, $\pi:\xi_t\to\gamma_t$ is also a covering of degree $d$ for each $t$. Recall that $P\in X^{\rm ev}$ and $\{ {\hat P}_1,\,{\hat P}_2\}\subset {\hat X}^{\rm ev}$. Since $\gamma_t$ goes to $P$ as $t\to\infty$ on the underlying topological surface ${\Bbb X}$ of $X$, using the homotopy lifting property again, we find that $d=1$ and $\xi_t$ goes to exactly one of ${\hat P}_1$ and ${\hat P}_2$.  Hence, we complete the proof of the claim. \\

Finally, we show that {\it $\omega$ has real periods.} Since $\omega\otimes \omega$ is a Strebel differential,  we know from \eqref{decomposition} that  $\omega$ gives a cellular decomposition of ${\hat X}$. Speaking in details, the 0-cells coincide with the zeroes of $\omega$,  the 1-cells are the critical horizontal trajectories of $\omega$, and the 2-cells are the topological discs centered at poles of $\omega$.  We look at the critical graph $G$ of $\omega$ as a embedded  graph on $X$. Since $\omega$ is a one-form, its critical graph $G$ has a natural orientation as in \cite[Section 4.1]{KZ2003}, where this oriented graph $G$  was called the {\it separatrix diagram} of $\omega$. Then the two topological spaces $X \setminus \{\text{poles of\ }\omega\}$ and $G$ have the same homotopy type. Since each edge of $G$ is a critical horizonal trajectory, the integral of $\omega$ over it gives a real number. By the homotopy equivalence between $X \setminus \{\text{poles of\ }\omega\}$ and $G$,  the integral of $\omega$ over a loop in $X \setminus \{\text{poles of\ }\omega\}$ equals that of $\omega$ over another loop in $G$, which is a real number.
\end{proof}

\nd {\bf Proof of Theorem \ref{thm:StrebelMetric}.}
By Lemma \ref{lem:RealPeriod}, $\omega$ has real periods on ${\hat X}$.  The projective function $f(z)=\exp\,\big(\int^z\,\sqrt{-1}\omega\big)$ on the surface of ${\hat X}$ punctured by the simple poles of $\omega$ has periods in ${\rm U}(1)$. Moreover, $f(z)$ extends to the simple poles of $\omega$ at which the \qd\ $\omega\otimes \omega$ have residues in $\mathbb{Z}$. Since $f(z)=F\circ\pi(z)$, the projective function $F(z)$ defined on $X\setminus D_q$ has monodromy in ${\rm U}(1)\rtimes {\Bbb Z}_2$ so that it becomes a developing map of a spherical metric on  $X\setminus D_q$.

At last we prove that {\it $F$ is compatible with $D_q$}. Choose a point $P$ in ${\rm Supp}\, D_q$.

\nd $\bullet$ Suppose that $P$ is a double pole of $q$ with residue $a>0$. Then there exists a complex coordinate chart $(U,\,z)$ centered at $P$ where $\sqrt{-q}(z)=a\,\frac{dz}{z}$. Hence, a branch of $F(z)$ equals $z^a$ up to a M{\" o}bius transformation in each sufficiently small disc near $P$ but not containing $P$. Hence the Schwarzian derivative $\{F,\,z\}$ of $F$ equals $\frac{1-a^2}{2z^2}$ near $P$.

\nd $\bullet$ Suppose that $P$ is a zero of order $k$ of $q$. As $k={\rm even}$, using the similar argument as above, we can show that there exists a complex coordinate chart $(U,\,z)$ centered at $P$, where a branch of $F(z)$ equals $z^{1+k/2}$ up to a M{\" o}bius transformation. As $k={\rm odd}$, we can also show that a branch of $F(z)$ equals $z^{1+k/2}$ up to a M{\" o}bius transformation in each sufficiently small disc near $P$ but not containing $P$.
$\hfill{\square}$\\

\nd {\bf Proof of Corollary \ref{cor:Strebel}.}  It follows from a combination of  Theorem \ref{thm:StrebelMetric} and Theorem \ref{thm:Strebel}. In particular, it follows from the preceding paragraph  that the $\ell$ extra conical singularities $Q_1,\cdots, Q_\ell$ coincide with the zeroes of the unique Strebel differential $q$ in Theorem  \ref{thm:Strebel}. $\hfill{\square}$\\

\begin{ex}
\label{ex:torus}
Denote by $E_\tau$ the elliptic curve $\frac{\Bbb C}{{\Bbb Z}\oplus {\Bbb Z}\tau}$ of modulus $\tau$ with $\Im\,\tau>0$.  We shall give two cone spherical metrics on the two special elliptic curves of moduli $\sqrt{-1}$ and $e^{\sqrt{-1}\pi/3}$, which have conical angles of $(2\pi\,a,\,4\pi)$ and $(2\pi\, a,\, 3\pi,\,3\pi)$ ($a>0$), respectively. Recall the {\it Weierstrass elliptic function}
\[ \wp(z)=\frac{1}{z^2}+\sum_{(0,0)\not=(m,n)\in {\Bbb Z}^2}\,\bigg(\frac{1}{(z-m-n\tau)^2}-\frac{1}{(m+n\tau)^2}\bigg)=\frac{1}{z^2}+\frac{g_2 z^2}{20}
+\cdots, \]
where $g_2=\sum_{(0,0)\not=(m,n)\in {\Bbb Z}^2}\,\frac{60}{(m+n\tau)^4}$. Let $\tau$ be $\sqrt{-1}$ or $e^{\sqrt{-1}\pi/3}$. Mulase-Penkava \cite[Example 4.5]{Mul98} used $\wp(z)$ to construct two Strebel differentials on the two elliptic curves of $E_\tau$'s.  Precisely speaking, they showed that $q=-\wp(z)\,dz^2$ are Strebel differentials on $E_\tau$'s, and it has a double zero at $(1+\tau)/2$ as $\tau=\sqrt{-1}$ and has the two simple zeroes of $(1+\tau)/3$ and $(2+2\tau)/3$ as $\tau=e^{\sqrt{-1}\pi/3}$. The critical graphs of these two Strebel differentials were drawn in \cite[Figures 4.6, 4.8]{Mul98}. Appying Theorem \ref{thm:StrebelMetric} to a positive multiple of $q$, we obtain two cone spherical metrics: the one on $E_{\sqrt{-1}}$ has conical singularities at $0$ and $(1+\sqrt{-1})/2$ with angles $2\pi\,a$ and $4\pi$, respectively, the other on $E_{e^{\sqrt{-1}\pi/3}}$ has conical singularities at $0$, $\big(1+e^{\sqrt{-1}\pi/3}\big)/3$ and $\big(2+2e^{\sqrt{-1}\pi/3}\big)/3$ with angles $2\pi\,a$, $3\pi$ and $3\pi$, respectively. These two cone sphercial metrics show that the existence result in Corollary \ref{cor:quasi} does depend on the moduli of Riemann surfaces.

The existence of the above cone spherical metric with angles $2\pi\,a$ and $4\pi$ on $E_{\sqrt{-1}}$ might be compared with a recent result by Z.-J. Chen and C.-S. Lin \cite{CL2017} that there exists no cone spherical metric with one conical singularity of angle in $\{6\pi,\, 10\pi,\, 14\pi,\,\cdots\}$ on each elliptic curve with modulus being purely imaginary, the two special cases $6\pi$ and $10\pi$ of which had been verified in \cite{CLW14} and \cite{CKL2016}, respectively.
\end{ex}

\section{Proof of Theorem \ref{thm:SimpleZeros}}
\label{sec:SimpleZeros}
\begin{defn}
\label{defn:tri}
{\rm We call a graph {\it trivalent} if and only if each vertex of the graph is incident to three half edges. For example, both the left graph and the right one in Example \ref{ex:3Poles} are trivalent.}
\end{defn}

By the correspondence between metric ribbon graphs and Strebel differentials \cite[Theorem 5.1]{Mul98}, it suffices to construct the corresponding trivalent metric ribbon graphs on a compact connected oriented topological surface of genus $g$. We shall divide our construction into two cases of $g=0$ and $g>0$. Case $g=0$ of the theorem is implied by Example \ref{ex:3Poles} and the following

\begin{lem}
 \label{Lem:SimpleZero}
    For $n>3$ and an arbitrary real vector $\alpha = (a_1, a_2, \cdots, a_n) \in (\mathbb{R}_{>0})^n$, there exists a Strebel differential $q$ on the Riemann sphere which has $n$ double poles with residue vector $\alpha$ and $2n-4$ simple zeroes.
\end{lem}

\begin{proof} It suffices to construct a trivalent connected metric graph $\Gamma$ on the two-sphere ${\Bbb S}^2$ such that

\nd $\bullet$ $\Gamma$ decomposes ${\Bbb S}^2$ into $n$ topological discs, say $D_1, D_2, \cdots$ and $D_n$;

\nd $\bullet$ the boundary $\partial D_j$ of $D_j$ has length equal to $a_j$ for each $1\leq j\leq n$.

\nd Let $G$ be a connected metric ribbon graph on ${\Bbb S}^2$  which decomposes  ${\Bbb S}^2$ into several discs. We call the vector formed by the lengths of the boundaries of these discs the {\it residue vector} of $G$. Note that if  $\alpha$ is the residue vector of $G$,  then the residue vector of the Strebel differential corresponding to $G$ is $\frac{\alpha}{2\pi}$. We may assume that $a_1\leq a_2\leq \cdots \leq a_n$ without loss of generality and decompose the construction of $\Gamma$ into the following two steps.\\

\nd {\bf Step 1.}  {\it We construct a metric ribbon graph $G_1$ with residue vector $(a_1,\cdots, a_n)$. However, $G_1$ is not necessarily trivalent.}
       \begin{itemize}
      \item[Case 1.] Assume $a_1 < a_2$. Then we define
      $l_1:=a_1>0,\,
        l_2:=a_2-a_1>0,\,
        l_3:=a_3-a_2+a_1> 0, \,
            \cdots,\,
       l_{n-2}:=a_{n-2}-a_{n-3}+ \cdots+(-1)^{n-1}a_1>0.$
    Since $l_{n-2} < a_{n-1} \leq a_n$, there exists a metric graph $G_0$  with residue vector $(l_{n-2}, a_{n-1}, a_n)$ by Example \ref{ex:3Poles}. Recall that in the example there are three graphs. $G_0$ might be any one of them according to the size relationship between the two quantities $l_{n-2}+a_{n-1}$ and $a_n$. In detail,  if $l_{n-2} + a_{n-1} = a_{n}$, then $G_{0}$ has only one vertex with valency $4$ and coincide with the middle graph in Example \ref{ex:3Poles}. Otherwise, $G_{0}$ is a trivalent graph corresponding to  either the left graph or the right one in the example. We embed $G_0$ into the plane ${\Bbb R}^2\subset{\Bbb S}^2$ so that $G_0$ decomposes ${\Bbb R}^2$ into a punctured disc and two discs. We take one of them, say $D$, such that its boundary $\partial D$ is a circle of length $l_{n-2}$. Taking a point $p$ on $\partial D$ other than any vertices of $G_0$, we could draw $n-3$ circles passing through $p$ such that the former circle contains the latter one in $D$ with the lengths $l_{n-3}, \cdots, l_2, l_1$, respectively, as the left picture in Figure \ref{fig:Construct}. Then we obtain a planar metric ribbon graph $G_1$ whose residue vector coincides with $(a_1,\cdots, a_n)$. Moreover, $G_0$ is a subgraph of $G_1$.

    \item[Case 2] Assume $a_1=a_2$. Choosing a sufficiently small positive number $\delta>0$, we draw at first the smallest circle with the circumference of $l_2=a_2 - 2\delta=a_1-2\delta$ and then take an edge of length $l_1=\frac{a_1}{2} + \delta$ which divides the circle equally. Then in the circle we obtain two new topological discs, the boundary of each of which has length $a_{1}$. Defining
    $
        l_3:= a_3 - a_2 + 2\delta>0,\,
        l_4:= a_4 - a_3 + a_2 - 2\delta>0,\,
            \cdots,$
    \[ l_{n-2}:=a_{n-2}-a_{n-3}+\cdots+(-1)^{n-2}a_2+(-1)^{n-1}\,2\delta>0, \]
     we find that $l_{j} + l_{j-1} = a_{j}$ for $3 \leq j \leq n-2$. By a similar construction as Case 1 as the right picture in Figure \ref{fig:Construct}, we obtain a planar metric ribbon graph $G_1$ with residue vector
       $(a_1,\cdots, a_n)$ and $G_0$ is a subgraph of it.
    \end{itemize}
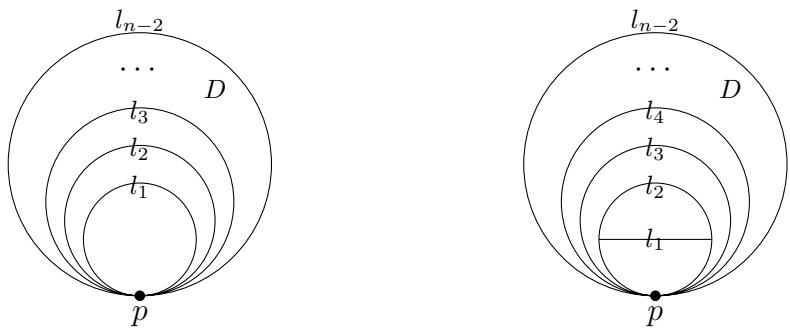
\begin{figure}[H]
\centering
\begin{minipage}[b]{.5\textwidth}
\centering
\begin{tikzpicture}
  \draw (0,0) circle (0.75)
            (0,0.25) circle (1)
            (0,0.5) circle (1.25)
            (0,1) circle (1.75);
  \node at (0,2.25) {$\cdots$};
  \node at (0, 0.7) {\footnotesize{$l_1$}};
  \node at (0, 1.2) {\footnotesize{$l_2$}};
  \node at (0, 1.7) {\footnotesize{$l_3$}};
  \node at (0, 2.9) {\footnotesize{$l_{n-2}$}};
  \node at (1, 2) {\footnotesize{$D$}};
  \fill (0, -0.75) circle (2pt) node[below] {$p$};
\end{tikzpicture}
\end{minipage}%
\begin{minipage}[b]{.5\textwidth}
\centering
\begin{tikzpicture}
  \draw (0,0) circle (0.75)
            (-0.75,0)--(0.75,0)
            (0,0.25) circle (1)
            (0,0.5) circle (1.25)
            (0,1) circle (1.75);
  \node at (0,2.25) {$\cdots$};
  \node at (0, 0) {\footnotesize{$l_1$}};
  \node at (0, 0.7) {\footnotesize{$l_2$}};
  \node at (0, 1.2) {\footnotesize{$l_3$}};
  \node at (0, 1.7) {\footnotesize{$l_4$}};
  \node at (0, 2.9) {\footnotesize{$l_{n-2}$}};
  \node at (1, 2) {\footnotesize{$D$}};
  \fill (0, -0.75) circle (2pt) node[below] {$p$};
\end{tikzpicture}
\end{minipage}%
\caption{The part inside $D$ of the two graphs (Case 1 left and Case 2 right).}
\label{fig:Construct}
\end{figure}

\nd {\bf Step 2.} {\it We modify the graph $G_1$ in Step 1 to  a trivalent graph, say $\Gamma$,  without change the residue vector.} Denote by $E_1, E_2, \cdots, E_{n-2}$  the edges (circles) of lengths  $l_1, l_2, \cdots, l_{n-2}$ in Figure \ref{fig:Construct}, respectively. Then we move $E_i$ along $E_{i+1}$ such that each vertex inside $D$ have valency at most $4$ and keep  the residue vector invariant as Figure \ref{fig:Breaking}. Denote by $G_2$ the new graph constructed from $G_1$ via this procedure.
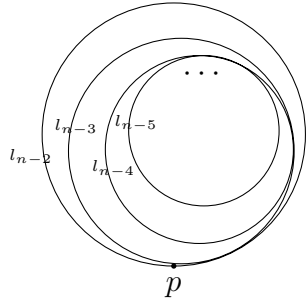
\begin{figure}[H]
  \centering
  \begin{tikzpicture}
     \draw (0,0) circle (1.75)
               (0.1,-0.22) circle (1.5)
               (0.34,-0.2) circle (1.25)
               (0.4,0.05) circle (1);
    \node at (0.4,0.8) {$\cdots$};
    \node at (-0.5, 0.15) {\tiny{$l_{n-5}$}};
    \node at (-0.8, -0.45) {\tiny{$l_{n-4}$}};
    \node at (-1.3, 0.1) {\tiny{$l_{n-3}$}};
    \node at (-1.9, -0.3) {\tiny{$l_{n-2}$}};
    \fill (0, -1.75) circle (1pt) node[below] {$p$};
  \end{tikzpicture}
  \caption{The graph after moving circles inside $D$.}
  \label{fig:Breaking}
\end{figure}

Suppose $l_{n-2}+a_{n-1} = a_n$. Then we show the initial  modification of $G_2$ as follows. Recall that in this case the subgraph $G_0$ of $G_2$ has a unique vertex $P$ of valency $4$ by Example \ref{ex:3Poles}. Denote by $E_{n-1}$ the circle of length $a_{n-1}$ of $G_0$. Choosing a sufficiently small $\varepsilon>0$, we at first decrease the length of $E_{n-2}$ to $l_{n-2}-2\varepsilon$ and break up the vertex $P$ into two vertices of valency $3$ by replacing it by a segment of length $\varepsilon>0$ connecting the two circles of $E_{n-1}$ and $E_{n-2}$; then we do the same thing for the vertex which is incident on the two circles $E_{n-2}$ and $E_{n-3}$ as  the middle picture in Figure \ref{fig:Surgery}. At last, we use the induction argument to complete the modification and obtain $\Gamma$. Suppose that the valency of each vertex on edges $E_{n-1}, E_{n-2}, \cdots, E_{i+1}$ is $3$ and there exists a vertex $v$ of valency $4$ on $E_i$. Taking a sufficiently small $\varepsilon_{i} >0$, we can break up $v$ into $2$ vertices, and connect them by a segment $C_{i}$ of length $\varepsilon_{i}$. Then we decrease the length of $E_i$ to $l_{i} - 2\varepsilon_{i}$ and increase the length of $C_{i+1}$ to $\varepsilon_{i+1} + \varepsilon_{i}$ as in Figure \ref{fig:induction}, where $\varepsilon_{n-2} = \varepsilon$. Doing the breaking-up procedure to the other vertices inside $D$ also as in Figure \ref{fig:induction}, we complete the induction argument and obtain a trivalent graph $\Gamma$ with residue vector $(a_1,\cdots, a_n)$.

\begin{figure}[H]
  \begin{center}
    \begin{tikzpicture}
      \draw (-4,1)--(-4,-1)
                (-4.4,1)--(-4.4,1.4);
      \draw (-4,1) arc (45:315:0.5cm and 1.4cm);
      \draw (-4,-1) arc (-135:135:0.5cm and 1.4cm);
      \draw (-4.4,0) ellipse (0.25 and 1);
      \node[rotate=90] at (-3.8,0) {\tiny{$\frac{l_{n-2}+a_{n-1}-a_n}{2}$} - $\varepsilon$};
      \node[rotate=90] at (-5.1,0) {\tiny{$\frac{l_{n-2}+a_n-a_{n-1}}{2}$} - $\varepsilon$};
      \node[rotate=90] at (-2.9,0) {\tiny{$\frac{a_{n-1}+a_n-l_{n-2}}{2}$} + $\varepsilon$};
      \node at (-4.3,1.2) {\tiny{$\varepsilon$}};
      \node at (-4.4,0.5) {$\vdots$};
      \node at (-4.6,-.3) {\footnotesize$D$};
      \node at (-4.4,-1) {\tiny{$l_{n-3}$}};
      \node at (-4,-1.7) {$l_{n-2} + a_{n-1} > a_n$};

      \draw plot[smooth] coordinates {(-0.1,0) (-0.5,1) (-1,0) (-0.5,-1) (-0.1,0)};
      \draw plot[smooth] coordinates {(0.1,0) (0.5,1) (1,0) (0.5,-1) (0.1,0)};
      \draw plot[smooth] coordinates {(-0.3,0) (-0.5,0.7) (-0.8,0) (-0.5,-0.6) (-0.3,0)};
      \draw (-0.5,0.7)--(-0.5,1);
      \draw (-0.1,0)--(0.1,0);
      \node[rotate=60] at (-1,0.6) {\tiny{$l_{n-2} - 2\varepsilon$}};
      \node[rotate=300] at (1,0.6) {\tiny{$a_{n-1}$}};
      \node at (0,0.1) {\tiny $\varepsilon$};
      \node at (-0.55,0.8) {\tiny $\varepsilon$};
      \node at (-0.5,0.4) {$\vdots$};
      \node at (-0.7,-0.2) {\footnotesize $D$};
      \node[rotate=60] at (-0.42,-0.55) {\tiny{$l_{n-3}$}};

      \node at (0,-1.7) {$l_{n-2} + a_{n-1} = a_n$};

      \draw (3.25,0)--(4.75,0)
               (2.75,1)--(2.75,1.4);
      \draw (5.25,0) ellipse (0.5 and 1.4)
                (2.75,0) ellipse (0.5 and 1.4)
                (2.75,0) ellipse (0.25 and 1);
      \node[rotate=90] at (2.1,0) {\tiny{$l_{n-2} - 2\varepsilon$}};
      \node[rotate=90] at (5.9,0) {\tiny{$a_{n-1}$}};
      \node[rotate=90] at (4,0.2) {\tiny{$\frac{a_{n}-l_{n-2}-a_{n-1}}{2} + \varepsilon$}};
      \node at (2.7,1.2) {\tiny{$\varepsilon$}};
      \node at (2.75,0.5) {$\vdots$};
      \node at (2.5,-.5) {\footnotesize $D$};
      \node at (2.75,-1) {\tiny{$l_{n-3}$}};
      \node at (4,-1.7) {$l_{n-2} + a_{n-1} < a_n$};
    \end{tikzpicture}
  \end{center}
\caption{The initial modificative graphs by breaking up of vertices.}
\label{fig:Surgery}
\end{figure}
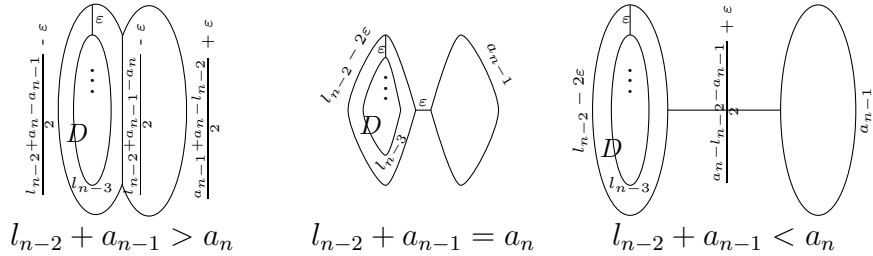

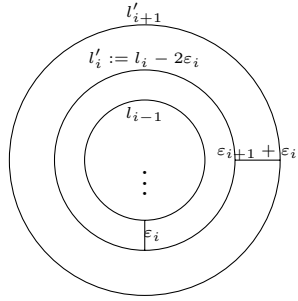
\begin{figure}[H]
\centering
\begin{tikzpicture}
  \draw (0,0) circle (1.8)
            (1.8,0)--(1.2,0)
            (0,0) circle (1.2)
            (0,-1.2)--(0,-0.8)
            (0,0) circle (0.8);
 \node at (1.5,0.1) {\tiny{$\varepsilon_{i+1} +\varepsilon_{i}$}};
 \node at (0.1,-1) {\tiny{$\varepsilon_{i}$}};
 \node at (0,1.95) {\tiny{$l'_{i+1}$}};
  \node at (0,1.35) {\tiny{$l_i':=l_i-2\varepsilon_{i}$}};
  \node at (0,0.6) {\tiny{$l_{i-1}$}};
  \node at (0,-0.2) {$\vdots$};
\end{tikzpicture}
\caption{Breaking up of vertices by induction.}
\label{fig:induction}
\end{figure}

The other two cases of $l_{n-2}+a_{n-1}>a_n$ and $l_{n-2}+a_{n-1}<a_n$ could be proved in a similar way. In detail,  we at first do the initial modification process for $G_2$ as the left and right pictures, respectively,  in  Figure \ref{fig:Surgery}. Then we apply the same induction argument as in Figure \ref{fig:induction} to obtain $\Gamma$.
\end{proof}

 For the case $g>0$, by the correspondence between Strebel differentials and metric ribbon graphs, we need to show
 \begin{lem}
 \label{Lem:HighGenus}
   Let $g$ and $n$ be two positive integers and $\alpha=(a_{1}, a_{2}, \cdots, a_{n})$ a vector in $(\mathbb{R}_{>0})^{n}$. Then there exist a compact oriented topological surface of genus $g$ and a connected trivalent metric ribbon graph on it which decomposes the surface into $n$ discs and has residue vector $\alpha$.
 \end{lem}
 \begin{proof}
We prove case $n>3$ by using Lemma \ref{Lem:SimpleZero} at first, and then show the three cases of $n=1,2$ and $3$ one by one.\\

\noindent {\bf Case 1.} Suppose $n>3$. Note that a compact connected oriented surface is homeomorphic to the two-sphere with $g$ handles. We prove the existence of the graph and the surface by induction on $g$. Assume that the result is valid for $g-1\geq 0$. That is, for any given $\alpha = (a_1, a_2, \cdots, a_n) \in (\mathbb{R}_{>0})^n$ and a sufficiently small $\epsilon >0$ such that $a_1-4\epsilon>0$, there exists a connected trivalent metric ribbon graph $\Gamma$ on an oriented compact surface of genus $g-1\geq 0$ such that $\Gamma$ decomposes the surface into $n$ topological discs and has residue vector $(a_1-4\epsilon, a_2, \cdots, a_n)$. As in Figure \ref{handle_graph}, we add a handle on the topological disc $D_1$ on the surface corresponding to the first residue $a_1-4\epsilon$  and draw two curves of lengths $\epsilon$ on the handle. Therefore, adding this handle endowed with these two curves to the original surface of genus $g-1$ endowed with the graph $\Gamma$, we obtain a new trivalent graph $\Gamma'$ with residue vector $\alpha$ embedded in a new oriented compact surface of genus $g$. We are done for this case.

\begin{figure}[H]
  \centering
  \begin{tikzpicture}
     \draw[gray, dashed] (-1,0) ellipse (0.5 and 1);
     \draw[gray] (-0.62,0.65) arc (40.5:319.5:0.5 and 1);
     \draw (-1,-0.35) arc (-160:160:2 and 1)
               (1,0) ellipse (0.5 and 0.2);
      \draw[gray] (-0.6,-0.25) arc (-150:150:1.6 and 0.6)
                        plot[smooth] coordinates {(-1.49,0.2) (-0.9,0.2) (-0.6,0.35)}
                        plot[smooth] coordinates {(-1.49,-0.2) (-0.9,-0.1) (-0.6,-0.25)}
                        (1,0.64) parabola[bend at end] (1.4,0.14)
                        plot[smooth] coordinates{(2.2,-0.7) (2.1,-0.8) (2,-0.8) (1.6,-0.75) (1.1,-0.53)} ;
    \draw[gray,dashed] plot[smooth] coordinates{(1.4,0.14) (1.7,0) (2,-0.3) (2.2,-0.7) (2.1,-0.8)} ;
    \node at (-1,1.1) {\tiny{$a_1 - 4\epsilon$}};
    \node at (-1.2,0.5) {\small{$D_1$}};
    \node at (1.9,0) {\tiny{$\epsilon$}};
    \node at (0.1,0.5) {\tiny{$\epsilon$}};
  \end{tikzpicture}
  \caption{Draw two curves on the handle.}
  \label{handle_graph}
\end{figure}
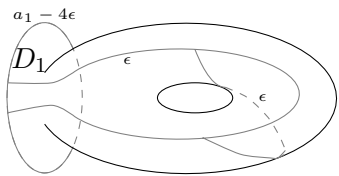

\noindent {\bf Case 2.} Suppose $n=3$.
By the adding-handle technique in Case 1, we only need to deal with the $g=1$ case.
Choose a sufficiently small $\epsilon >0$ such that $a_1-4\epsilon +a_2 \neq a_3$. By Example \ref{ex:3Poles}, there exists on ${\Bbb S}^2$ a trivalent metric ribbon graph with residue vector $(a_1-4\epsilon, a_2, a_3)$. By the adding-handle technique as Case 1, we could obtain the desired graph on a topological torus.\\

\noindent {\bf Case 3.} Suppose $n=2$.
Similarly, we only need to consider the $g=1$ case. Taking the following trivalent metric ribbon graph on a torus
\begin{figure}[H]
\begin{center}
\begin{tikzpicture}
\draw (0,0) ellipse (2.4 and 1.4);
\begin{scope}[scale=.8]
\path[rounded corners=24pt] (-.9,0)--(0,.6)--(.9,0) (-.9,0)--(0,-.56)--(.9,0);
\draw[rounded corners=28pt] (-1.1,.1)--(0,-.6)--(1.1,.1);
\draw[rounded corners=24pt] (-.9,0)--(0,.6)--(.9,0);
\end{scope}
\draw[gray] (0,0) ellipse (1.6 and .9);
\draw [gray, dashed] plot [smooth, tension=1] coordinates { (-0.5,-0.85) (-0.2,-0.17) (0.4,-1.35) (0.55,-0.85)};
\begin{scope}
  \clip (-0.55,-0.9) rectangle (-0.18,-0.15);
  \draw [gray] plot [smooth, tension=1] coordinates { (-0.5,-0.85) (-0.2,-0.17) (0.4,-1.35) (0.55,-0.85)};
\end{scope}
\begin{scope}
  \clip (0.35,-1.4) rectangle (0.6,-0.8);
  \draw [gray] plot [smooth, tension=1] coordinates { (-0.5,-0.85) (-0.2,-0.17) (0.4,-1.35) (0.55,-0.85)};
\end{scope}

\draw [gray, dashed] plot [smooth, tension=1] coordinates { (-0.5,0.85) (-0.3,1.35) (0.3,.2) (0.5,.85)};
\begin{scope}
  \clip (-0.55,0.8) rectangle (-0.25,1.4);
  \draw [gray] plot [smooth, tension=1] coordinates { (-0.5,0.85) (-0.3,1.35) (0.3,.2) (0.5,.85)};
\end{scope}
\begin{scope}
  \clip (0.25,.15) rectangle (0.55,.9);
  \draw [gray] plot [smooth, tension=1] coordinates { (-0.5,0.85) (-0.3,1.35) (0.3,.2) (0.5,.85)};
\end{scope}
\node at (-.1,.55) {\tiny{$a_2$}};
\node at (.2,-.55) {\tiny{$a_1$}};
\node at (-0.1,-1) {\tiny{$b_4$}};
\node at (0.2,1) {\tiny{$b_3$}};
\node at (-1.7,0) {\tiny{$b_1$}};
\node at (1.8,0) {\tiny{$b_2$}};

\end{tikzpicture}
\caption{The trivalent graph on a torus with two double poles.}
\end{center}
\end{figure}
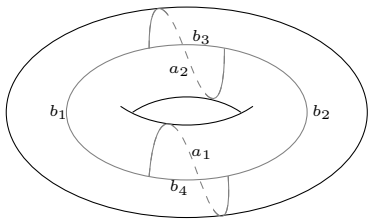
\nd we find that the residue vector $(2b_1 + a_1 + a_2 + b_3+b_4, 2b_2 + a_1 + a_2 +b_3 + b_4)$ of the graph could achieve each vector in $(\mathbb{R}_{>0})^2$ as $a_i$'s and $b_j$'s vary in ${\Bbb R}_{>0}$. \\

\noindent {\bf Case 4.} Suppose $n=1$. Then we are done by
taking the following trivalent metric ribbon graph with residue $2(a_1+b_1+b_2)$.
\begin{figure}[H]
\begin{center}
\begin{tikzpicture}
\draw (0,0) ellipse (2.4 and 1.4);
\begin{scope}[scale=.8]
\path[rounded corners=24pt] (-.9,0)--(0,.6)--(.9,0) (-.9,0)--(0,-.56)--(.9,0);
\draw[rounded corners=28pt] (-1.1,.1)--(0,-.6)--(1.1,.1);
\draw[rounded corners=24pt] (-.9,0)--(0,.6)--(.9,0);
\end{scope}

\draw[gray] (0,0) ellipse (1.6 and .9);
\draw [gray, dashed] plot [smooth, tension=1] coordinates { (-0.5,-0.85) (-0.2,-0.17) (0.4,-1.33) (0.55,-0.85)};
\begin{scope}
  \clip (-0.55,-0.9) rectangle (-0.18,-0.15);
  \draw [gray] plot [smooth, tension=1] coordinates { (-0.5,-0.85) (-0.2,-0.17) (0.4,-1.33) (0.55,-0.85)};
\end{scope}
\begin{scope}
  \clip (0.35,-1.4) rectangle (0.6,-0.8);
  \draw [gray] plot [smooth, tension=1] coordinates { (-0.5,-0.85) (-0.2,-0.17) (0.4,-1.33) (0.55,-0.85)};
\end{scope}
\node at (.2,-.55) {\tiny{$a_1$}};
\node at (-0.1,-1) {\tiny{$b_2$}};
\node at (0.2,1) {\tiny{$b_1$}};
\end{tikzpicture}
\caption{The trivalent graph on a torus with a double pole.}
\end{center}
\end{figure}
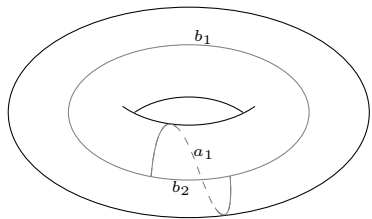
\end{proof}

\section{Cone angles of cone spherical metrics on the Riemann sphere}
\label{sec:sphere}
Reducible (irreducible) metrics are also called cone spherical metrics with {\it coaxial} ({\it non-coaxial}) holonomy by
Mondello-Panov \cite{MP1505} and Eremenko \cite{Er2017}. On the Riemann sphere  $\overline{\Bbb C}$, Mondello-Panov \cite[Theorems A and C]{MP1505} found a necessary and sufficient condition, say Condition (MP), satisfied by cone angles of irreducible metrics on $\overline{\Bbb C}$. Eremenko \cite[Theorem 1]{Er2017} obtained a corresponding angle condition, say Condition (E), for reducible metrics on $\overline{\Bbb C}$.  A combination of (MP) and (E) gives the complete information for cone angles of cone spherical metrics on $\overline{\Bbb C}$. We observe that the two conditions of (MP) and (E) do have an intersection in general. For example, the five cone angles of $6\pi, 3\pi, 3\pi, 7\pi/2, 7\pi/2$ satisfy both (MP) and (E) on $\overline{\Bbb C}$. We obtained in Corollaries \ref{cor:Strebel} and \ref{cor:quasi} a special class of cone spherical metrics constructed from Strebel differentials on compact Riemann surfaces, where the cone angles are explicitly given. We naturally ask on the Riemann sphere which one of the two conditions of (E) and (MP) the cone angles of these metrics in these two corollaries would satisfy, and give a partial answer as follows. \\

$\bullet$  If the cone spherical metrics in Corollary \ref{cor:Strebel} are defined by degenerate Strebel differentials on ${\overline {\Bbb C}}$, then by Remark \ref{rem:red} they are reducible and their cone angles satisfy Condition (E). On the other hand, as $a_1,\cdots, a_n$ are non-integers and $m_1,\cdots, m_\ell$ are odd integers, the cone spherical metrics  on ${\overline {\Bbb C}}$  in the corollary must be irreducible and their cone angles  satisfy Condition (MP) but not (E). In fact, it was shown in \cite[Theorem 1.4]{CWWX2015} that if a reducible metric on ${\overline {\Bbb C}}$ has more than two conical singularities, then at least one cone angle of the metric lies in $2\pi{\Bbb Z}_{>1}$.
\\

$\bullet$ We observe that the cone spherical metrics in Corollary \ref{cor:quasi} are always defined by non-degenerate Strebel differentials since the differentials there have simple zeroes.  If one of $a_1,\,a_2,\cdots, a_n$ is not an integer, then the cone spherical metrics on ${\overline {\Bbb C}}$ in the corollary have cone angles satisfying Condition (MP), which was also pointed out to us by one of the referees.  \\

$\bullet$ At last, we consider the cone spherical metrics on ${\overline {\Bbb C}}$  in Corollary \ref{cor:quasi} such that {\it $a_1,a_2,\cdots, a_n$ are all integers.} As $n=3$, Since the metrics have exactly two non-integral angles equal to $3\pi$, their monodromy groups is isomorphic to $\mathbb{Z}_2$. Hence, they are reducible and their angles satisfy Condition (E) but not Condition (MP) by a simple computation. As $n\geq 4$, by a simple computation, the metrics in the corollary always have cone angles satisfying Condition (MP). However, their angles do not necessarily satisfy Condition (E). For example,  (E) is satisfied by the six cone angles $4\pi, 4\pi, 3\pi,3\pi,3\pi,3\pi$ of the metrics corresponding to $(a_1,a_2,a_3,a_4)=(1,1,2,2)$ satisfy (E), but not by the four cone angles $3\pi,3\pi,3\pi,3\pi$ of the metrics corresponding to $(a_1,a_2,a_3,a_4)=(1,1,1,1)$. Recall that conical singularities with cone angle $2\pi$ are just marked smooth points of the metrics.

\section{Fundings}
The authors express their deep gratitude to the anonymous referees for their valuable suggestions and questions to the old version of the manuscript, according to which the authors improved its exposition greatly. In particular, the authors added a short section (Section \ref{sec:sphere}) to answer two questions by one of them. The first author is supported in part by the National Natural Science Foundation of China (grant no. 11426236), and the third by China Scholarship Council, and the last author  by the National Natural Science Foundation of China (grant no. 11571330) and the Fundamental Research Funds for the Central Universities.

\end{document}